\documentclass[12pt]{amsart}

\usepackage{amsmath,amssymb,amsthm,amscd,amsfonts,amsbsy,bm}
\usepackage[english]{babel}
\usepackage[mathscr]{eucal}

\usepackage{graphicx}

\usepackage{xcolor}

\font\sc=rsfs10 at 12pt

\numberwithin{equation}{section}

\renewcommand{\a}{\alpha}
\renewcommand{\b}{\beta}
\newcommand{\g}{\gamma}

\renewcommand{\d}{\delta}

\newcommand{\e}{\epsilon}

\newcommand{\z}{\zeta}

\renewcommand{\th}{\theta}
\newcommand{\vt}{\vartheta}

\renewcommand{\k}{\kappa}

\renewcommand{\l}{\lambda}
\renewcommand{\L}{\Lambda}

\newcommand{\n}{\nu}
\newcommand{\x}{\xi}

\renewcommand{\r}{\rho}

\renewcommand{\t}{\tau}

\newcommand{\vf}{\varphi}

\renewcommand{\o}{\omega}
\renewcommand{\O}{\Omega}

\newcommand{\vs}{\varsigma}

\newcommand{\C}{{\mathbb C}}
\newcommand{\R}{{\mathbb R}}
\newcommand{\Z}{{\mathbb Z}}

\newcommand{\fb}{{\mathbf f}}
\newcommand{\gb}{{\mathbf g}}
\newcommand{\hb}{{\mathbf h}}

\newcommand{\tb}{{\mathbf t}}
\newcommand{\ub}{{\mathbf u}}
\newcommand{\vb}{{\mathbf v}}

\newcommand{\Fb}{{\mathbf F}}
\newcommand{\Gb}{{\mathbf G}}

\newcommand{\Pb}{{\mathbf P}}

\newcommand{\Tb}{{\mathbf T}}

\newcommand{\Wb}{{\mathbf W}}

\newcommand{\bfb}{\mathfrak b}

\newcommand{\cF}{\mathfrak c}

\newcommand{\DF}{\mathfrak D}

\newcommand{\gF}{\mathfrak g}
\newcommand{\GF}{\mathfrak G}
\newcommand{\HF}{\mathfrak H}

\newcommand{\kF}{\mathfrak k}
\newcommand{\KF}{\mathfrak K}

\newcommand{\pF}{\mathfrak p}
\newcommand{\PF}{\mathfrak P}

\newcommand{\Ac}{{\mathcal A}}
\newcommand{\Bc}{{\mathcal B}}

\newcommand{\Ec}{{\mathcal E}}

\newcommand{\Hc}{{\mathcal H}}

\newcommand{\Kc}{{\mathcal K}}
\newcommand{\Lc}{{\mathcal L}}

\newcommand{\Sc}{{\mathcal S}}

\newcommand{\Ss}{\sc\mbox{S}\hspace{1.0pt}}

\newcommand{\Dc}{\sc\mbox{D}\hspace{1.0pt}}

\DeclareMathOperator{\rank}{rank}

\newcommand{\Ran}{\hbox{{\rm Ran}}\,}

\newcommand{\dbar}{\overline{\partial}}

\newtheorem{theorem}{Theorem}[section]
\newtheorem{proposition}[theorem]{Proposition}
\newtheorem{lemma}[theorem]{Lemma}

\theoremstyle{definition}
\newtheorem{definition}[theorem]{Definition}

\theoremstyle{remark}

\begin{document}

\title[Weighted Estimates and Finite Rank Theorem]{Some weighted estimates for the $\dbar$- equation and  a finite rank theorem for Toeplitz operators in the Fock space.}
\author[Rozenblum]{Grigori Rozenblum}
\address{Department of Mathematics, Chalmers  University of Technology,   Gothenburg, Sweden}
\email{grigori@chalmers.se}
\author[Shirokov]{Nikolay Shirokov}
\address{Department of Mathematics, St.Petersburg State University, St.Petersburg, Russia}
\email{nikolai.shirokov@gmail.com}
\begin{abstract} We consider the $\dbar-$ equation in $\C^1$ in classes of functions  with Gaussian decay at infinity. We prove that if the right-hand side of the equation  is majorated by $\exp(-q|z|^2)$, with some positive $q$, together with derivatives up to some order, and is orthogonal, as a distribution, to all analytical polynomials, then there exists a solution with decays, together with derivatives, as $\exp(-q'|z|^2)$, for any $q'<q/e$. This result carries over to the $\dbar$-equation in classes of distributions, again, with Gaussian decay at infinity, in some precisely defined sense. The properties of the solution are used further on to prove the finite rank theorem for Toeplitz operators with distributional symbols in the Fock space: the symbol of such operator must be a combination of finitely many  $\delta$-distributions and their derivatives. The latter result generalizes the recent theorem on finite rank Toeplitz operators with symbols-functions.
\end{abstract}
\maketitle
\section{introduction}\label{intro}

\subsection{}  One of important results  in complex analysis is the theorem by L. H\"ormander (\cite{Hor}, Theor. 4.4.2) on  solvability and estimates for the $\dbar$ equation in weighted classes. This theorem, in application to the case of functions on  $\C^1$, states the following.
\begin{theorem} Let $\Wb=\Wb(z)$ be a subharmonic function on $\C^1$. Then, for any function $h(z)$, square integrable with weight $\exp(-\Wb(z))$, there exists a solution $u(z)$ of the equation $\dbar u=h$ such that
\begin{equation*}\label{TheHorm}
    \int_{\C}|u(z)|^2e^{-\Wb(z)}(1+|z|^2)^{-2}d\l(z)\le  C \int_{\C}|h(z)|^2e^{-\Wb(z)}d\l(z),
\end{equation*}
with some constant $C$, depending only on $\Wb$, where $\l$ is the Lebesgue measure on $\C$.
\end{theorem}
Usually, this theorem is applied in the case when the function $\Wb$ grows at infinity, so the weight $e^{-\Wb(z)}$ decays, rather fast, at infinity. Thus  the given function $h$ and the solution $u$ may grow  at infinity, with the restriction put on $u$ just slightly weaker than the ones for $h$.

We are interested in the opposite situation, when the weight $e^{-\Wb(z)}$ grows at infinity, so that  the function $\Wb(z)$ is not subharmonic. In other words, we look for  solutions of the $\dbar$-equation in some classes of functions decaying at infinity, provided that the given function $h$ on the right-hand side  decays at infinity as well.  It is clear that, unlike the H\"ormander case,   such solution may exist only if the obvious  necessary condition is fulfilled: the given function $h$ must be orthogonal to all analytical polynomials. The question is whether this necessary condition is a sufficient one.

The extreme case of the above setting is when the weight $e^{-\Wb(z)}$ is taken to be $+\infty$ outside some bounded  set $\O\subset \C$. More formally, this means that, given a function $h$ with compact support in $\O$, we are looking for the solution $u$ of the $\dbar$ equation $\dbar u=h$ such that $u$ is compactly supported as well.  The corresponding result seems to be a folklore one, it can be found, for example, in \cite{Gunning}, Lemma on  P.44, and in many other sources. A proof of this result, given in \cite{AlexRoz}, see Lemma 3.2 there, covers also its extension to    distributions with compact support, i.e., in $\Ec'(\C^1).$

\begin{lemma}\label{4.lemAlexRoz} Suppose that $\fb\in\Ec'(\C^1)$. Then the following two properties are equivalent:\\
(a) there exists a distribution $\gb\in \Ec'(\C^1) $ such that $\dbar \gb=\fb$, moreover, the support of $\gb$ is contained in the complement of the unbounded component of the complement of the support of $\fb$;\\
(b) $\fb$ vanishes on all analytical polynomials of $z$ variable, i.e. $\langle \fb,z^k\rangle = 0$ for all $k\in \Z_+$.
\end{lemma}

In the present paper we look for an extension of this result to the case when the  compact support condition is replaced by the   Gaussian decay one. Certain classes of distributions are introduced, formalizing the notion of the Gaussian decay, intermediate
between compactly supported and Schwartz spaces,  and for these classes a proper  analogy of Lemma \ref{4.lemAlexRoz} is proved, with the same necessary condition (b) which thus turns out to be sufficient.

Further on, we show that this procedure, solving the $\dbar$-equation, while controllably weakening the decay quality of the distributions in question, improves their local regularity, so, by means of iterating this procedure, after a finite number of steps, we arrive at a \emph{function} with Gaussian type decay.

\subsection{} When studying the above problem, we had in mind  a specific application arising from the theory of Toeplitz operators in the Fock space.

Such operators were  introduced by F.A. Berezin in \cite{Ber}, in the framework  of his general quantization program, and  were being extensively studied from different points of view further on, see, especially, \cite{Cob2} and the recent books \cite{Vasil} and \cite{Zhu}. These operators are often called 'Berezin-Toeplitz' and present a special case of Toeplitz operators in Bergman type spaces.

Generally, let $\Bc\subset L_2(\O)$, with respect to some measure, be a \emph{Bergman type space} consisting of solutions of some elliptic equation or system in a domain $\O$ in a real or complex Euclidean space. For  $\Bc$, the most common examples are the space of analytical functions in a bounded domain in $\C^d$ (say, disk, ball, polydisk) -- the classical Bergman space, as well as  the space of entire analytical functions in $\C^d$, square integrable with the Gaussian weight, -- the Bargmann-Fock space,  similar spaces of harmonic functions etc. Denote by $\Pb:L_2(\O)\to \Bc$ the orthogonal projection onto $\Bc$. For a function $F$ defined on $\O$, the Toeplitz operator $\Tb=\Tb(F)=\Tb(F;\Bc)$ is the operator in $\Bc$ acting as $\Bc\ni u\mapsto \Pb F u\in \Bc$. Here, $F$ is called the \emph{symbol} of the Toeplitz operator. This definition is unambiguous for the case of a bounded function $F$. However, the formula defining the action of the operator can be assigned an exact meaning also for certain unbounded functions $F$, for measures and even for some distributions. A detailed description of such Toeplitz operators  can be found in \cite{AlexRoz}, \cite{RozToepl}; we give more explanations below.

The properties of Toeplitz operators  in Bergman type spaces attract a considerable interest now, due to an expanding range of applications in Analysis and Mathematical Physics. One of questions that has been  under discussion recently is the one on finite rank operators.

The finite rank problem consists in the following. Suppose that for some symbol $F$, the operator $\Tb(F)$ has finite rank. What can be said about $F$ in this case? For $F$ being a function, the natural answer to expect is that if $\Tb(F)$ has finite rank then $F$ must be zero. For more general $F$, some nontrivial, but nevertheless, quite degenerate answers are possible.

Presently, this question has been under an active study. One can find a detailed historical overview in \cite{RozToepl}, \cite{RShir}, and \cite{BorRoz}. In particular, in \cite{BorRoz} a finite rank theorem has been proved for  operators in the Fock space on $\C^1$  with symbols-functions with a mild, almost sharp, growth restrictions imposed.
However, the reasoning in \cite{BorRoz} does not apply to symbols-distributions. The only presently known approach to deal with this latter case, developed for compactly supported symbols in
\cite{AlexRoz}, is based upon the result on the solvability  of the $\dbar$ equation, namely on Lemma \ref{4.lemAlexRoz}. Following this approach, with the compactness of support condition dropped, we thus need, as an important ingredient, to solve the $\dbar$ equation in some  classes of distributions with Gaussian decay.

\subsection{}
We start with introducing the spaces of  distributions and give a detailed description of  Toeplitz operators  with distributional symbols. Then we discuss the finite rank property and its relation to some infinite matrices.

  Our approach for the  extension of the finite rank theorem from functions to distributions is based upon a smoothness reduction: if the finite rank property holds for a certain symbol-distribution, then it holds for another symbol-distribution, less singular than the initial one. It is for this  reduction that we need some lengthy analysis of the properties of solutions of the $\dbar$-equation in classes of functions and distributions with Gaussian decay at infinity. In Section 4 we establish these estimates for functions, and in Section 5 we carry over these estimates to distributions. In Section 6, we present the proof of the finite rank theorem for the general case. 

It is known for compactly supported symbols, see, e.g., \cite{RozToepl}, that the finite rank property, once established for a Bergman type space of analytical functions, can be extended to some other Bergman type spaces. There are some specifics of that procedure when the compactness of support condition is dropped. We will deal with this topic, as well as more applications of the finite rank result, on some other occasion.

The authors express their gratitude to the  Mittag-Leffler Institute where they were given an excellent possibility to work on the paper.

\section{Toeplitz operators in the Fock space. Classes of symbols-distributions}\label{Sect2}
\subsection{Operators with bounded symbols}
We start this section by recalling some basic facts concerning the Fock space and operators there.

We identify the plane $\R^2$ with the complex plane $\C$ and denote by $\n$ the normalized Gaussian measure, $d\n= \o(z) d\l$, where $d\l$ is the two-dimensional Lebesgue measure, $\o(z)=\pi^{-1} e^{-|z|^2}.$ (We choose this version of the weight, rather than the alternative one $(2\pi)^{-1}e^{-|z|^2/2}$ in order to be in conformity, say,  with \cite{BauLe} and \cite{Zhu}.) In the space $\Hc=L_2(\C,d\n)$ we consider the subspace $\Bc$, the Fock space, which consists of entire analytical functions. By $(\cdot,\cdot)$ we will denote the scalar product in these spaces. The orthogonal projection $\Pb:\Hc\to \Bc$ is known to be an integral operator with smooth kernel,
\begin{equation*}\label{2.kernel}
    (\Pb u)(z)=\int_{\C}\k(z,w){u(w)}d\n(w)=(u,\k(\cdot,z)),
\end{equation*}
where
$\k(z,w)=e^{z\overline{w}}=\overline{\k(w,z)}$. In particular, if $u\in\Bc,$ we have $\Pb u =u$, or
\begin{equation}\label{2.repro}
    u(z)=\int_\C \k(z,w){u(w)}d\n(w)=(u,\k_z(\cdot)); \ \ \k_z(\cdot)= \k(.,z)=\overline{\k(z,\cdot)};\end{equation}
 equation \eqref{2.repro} is called the {reproducing property} and $\k(z,w)$ is called the \emph{reproducing kernel.}

For a  function $F$ defined on $\C$, the Toeplitz operator with symbol $F$ acts as an integral one,
\begin{equation*}\label{2.Oper}
    (\Tb(F)u)(z)=(\Pb F u)(z)=\int_{\C}\k(z,w)F(w){u(w)}d\n(w),
\end{equation*}
being defined on such functions $u\in\Bc$ for which $\k_z(\cdot) F u\in \Hc$ for almost all $z$ and $\Tb(F)u\in \Bc$.
If $F\in  L_\infty$, this operator is, obviously,  defined for all functions in $\Bc$ and bounded in $\Bc$, as a product of bounded operators. The operator's sesquilinear  form is
\begin{equation*}\label{2.form}
    \tb_F(u,v)=(\Tb(F)u,v)=\int_\C F(w)u(w)\overline{v(w)}d\n(w).
\end{equation*}
\subsection{Operators with unbounded symbols and symbols-distribu\-tions}
Our aim now is to define the Toeplitz operator for a larger class of symbols. There are several discussions of this topic in the literature, see, e.g., \cite{BauLe}, \cite{Ja1}, \cite{Ja2}, \cite{RozToepl}, \cite{RozSW} and references therein. These papers, however, consider the case of $F$ being a function (or, as in \cite{RozSW}, a measure) with certain growth limitations, or a distribution with compact support. We will gradually
extend the set of admissible symbols, to reach, finally, a certain class of non-compactly supported distributions.

If we drop the boundedness condition for $F$, the Toeplitz operator is not necessarily bounded, being defined on the set of functions $u\in\Bc$ satisfying  $\Tb(F)u\in \Bc.$ As in \cite{BauLe}, we introduce classes $\DF_\cF$, $\cF\le 1$ by
\begin{equation}\label{1.Dc}
    \DF_\cF=\{F:\C\to\C, |F(z)|\le \bfb e^{\cF|z|^2}\}
\end{equation}
 for  some $\bfb$. We also define $\DF_{1,-}$ as the space of functions $F$ satisfying $|F(z)|=O(e^{|z|^2-a|z|})$ for any $a>0$.

Generally, it is hard to describe explicitly  the domain of the Toeplitz operator with an unbounded symbol.
If $F\in \DF_\cF$, $\cF<1/2$, the domain of $\Tb(F)$ contains at least all functions $u\in \Bc\cap \DF_{1/2-\cF}$ and, in particular, is dense in $\Bc$. Under a less restrictive condition, $F\in \DF_{1,-}$, the Toeplitz operator is still densely defined and, in particular, its domain contains all analytical polynomials, as well as all reproducing kernels $\k_z$, $z\in\C$ and their finite linear combinations.  In the finite rank problem, which we mainly discuss in this paper, it is sufficient to consider the action of the operator on these dense subsets.
Reasonable extensions of the operator $\Tb(F)$ beyond \eqref{1.Dc} are discussed in \cite{Ja1}, \cite{Ja2}, and in \cite{BauLe}.

In the analysis of the finite rank problem, it is  convenient to consider the sesquilinear form $\tb=\tb_F$,
 \begin{equation}\label{2.sesqForm}
    \tb_F(u,v)=(F,\bar{u}{v})=\int_{\C} F(w){u(w)}\overline{v(w)} \o(w)d\l(w).
 \end{equation}

  If $F\in \DF_\cF, \cF<1$,  this form is defined at least on all functions $u,v\in \Bc\cap\DF_{\cF'}$, with $\cF'<(1-\cF)/2$.  This set is, again, dense in $\Bc$. If $F$ is a real function with constant sign, the sesquilinear form, thus defined, is closable and it corresponds to a self-adjoint operator. In the general case, there is no natural way to associate a \emph{closed} operator with the sesquilinear form \eqref{2.sesqForm}.  Nevertheless, for $F\in \DF_{1,-}$, the sesquilinear form \eqref{2.sesqForm} is consistent with the action of the operator $\Tb(F)$ at least on the functions $u,v$ being  the reproducing kernels or $u(w)=\k_z(w), v=\k_{z'}(w)$, or analytical polynomials $u(w)=p(w), v(w)=q(w), w\in\C$ :
  \begin{equation*}\label{2.form2}
    (\Tb(F)\k_z,\k_{z'})=\tb(\k_z,\k_{z'}),
 \end{equation*}
 and $(\Tb(F)p,q)=\tb(p,q)$.

We pass to the case of symbols-distributions.
 For $F\in\Ec'(\C)$,  i.e., with $F$ being  a distribution with compact support,  Toeplitz operators in $\Bc$ were, probably, first considered in \cite{AlexRoz}. Having two functions $u,v\in \Bc$, we can define the sesquilinear form generalizing \eqref{2.sesqForm}:
\begin{equation}\label{2.DistrE}
   \tb_F(u,v)=(\Tb(F) u,v)= \langle F, \o u\overline{v}\rangle= \langle\o F,u\overline{v}\rangle,
\end{equation}
where $\langle\cdot,\cdot\rangle$ denotes the standard action of the distribution on the function. This action is well defined since the product $u\overline{v}$ belongs to $\Ec(\C)$. One should keep in mind that in our notations, the parentheses $(\cdot,\cdot)$ have the weight factor $\o(z)$ incorporated in the measure, while the angle brackets $\langle \cdot ,\cdot \rangle$ correspond to the Lebesgue measure induced paring.

From a different point of view, the projection $\Pb$, possessing a smooth kernel, can be extended to a continuous operator $\widetilde{\Pb}: \Ec'(\C)\to \Ec(\C)$ by setting
\begin{equation*}\label{2.DistrExt}
    (\widetilde{\Pb} F)(z)=\langle F, \o(\cdot)\k_z(\cdot)\rangle.
\end{equation*}
Thus, the Toeplitz operator $\Tb(F)$ in $\Bc$ is represented as
\begin{equation}\label{2.Oper.Distr}
    \Tb(F) u= \widetilde{\Pb}uF,
\end{equation}
where $uF\in \Ec'(\C)$ is understood as the product of the function $u\in\Bc$ and the distribution $F\in\Ec'$.

Since all distributions with compact support have finite order,
we have for the sesquilinear form \eqref{2.DistrE}:
\begin{equation}\label{2.uv.Cl}
    |\tb_F(u,v)|\le C\|u\overline{v}\|_{C^l(K)}\le C'\|u\|_{C^l(K)}\|v\|_{C^l(K)}
\end{equation}
for a certain compact $K\subset \C$ and some integer $l$. For analytical functions $u,v$, the $C^l$- norms on the right-hand side in \eqref{2.uv.Cl} are bounded by their $\Bc$-norms. Therefore, the sesquilinear form  $\tb$ is bounded in the Hilbert space $\Bc$ and thus
the Toeplitz operator \eqref{2.Oper.Distr} is bounded as well. This circumstance was essentially used in \cite{AlexRoz}, \cite{RozToepl}. If the condition of compact support is dropped, this is not, generally, true, and we need to restrict ourselves to a special class of distributions, defined below, with a control of their behavior at infinity.

We define the class of functions with Gaussian growth, $\Dc_q=\Dc_q(\C)$, as consisting of such functions $\psi(z)\in\Ec(\C)$ that for any multi-index $\a=(\a_1,\a_2)$, the derivative $D^{(\a)}\psi= D_{1}^{\a_1}D_{2}^{\a_2}\psi$ satisfies
\begin{equation}\label{2.Dq}
    |D^{(\a)}\psi(z)|=o(\exp(q|z|^2))
\end{equation}
as $|z|\to\infty$

The system of constants $C(\a)=\sup_{z\in\C}\{ |D^{(\a)}\psi(z)|\exp(-q|z|^2)\}$ defines a locally convex topology in $\Dc_q$ in the usual way.

\begin{definition} The space of distributions $\Dc_q'$ is defined as the dual space to $\Dc_q$.
\end{definition}

Since the space of  smooth functions $\Dc_q$ satisfies the inclusions
$\Ss\subset\Dc_q\subset\Ec$,
where $\Ss$ is the Schwartz space of rapidly decaying functions, we have,
    $\Ec'\subset\Dc_q'\subset\Ss'.$

Along with $\Dc_q$, we consider a scale of Banach spaces of functions with finite smoothness, $\L_{q,l}$ consisting of $C^l$ - smooth functions subject to  estimate \eqref{2.Dq} for all $\a: |\a|\le l$, as well as the dual \emph{Banach} spaces $\L_{q,l}'$, with natural norms; $\Dc_q=\bigcap_l \L_{q,l}$, $\Dc_q'=\bigcup_l \L_{q,l}'$. It is important to keep in mind that for $\psi$ being \emph{a function}, the condition $\psi\in \L_{q,l}'$ imposes rather heavy decay conditions on $\psi.$ The spaces $\L_{q,l}$ are separable, with $\Dc(\C) $, the space of functions with compact support, dense in $\L_{q,l}$. The latter property implies, in particular,  that $\L_{q,l}$ is dense in $\L_{q',l}$ for $q<q'$.   In the standard way, $\Dc_{q}'$ turns out to be the inductive limit of the spaces $\L_{q,l}'$, so, similar to $\Ec'$ and $\Ss'$, any distribution in $\Dc_{q}'$ has finite order, i.e.,  for some $l$ it can be extended by continuity to a continuous linear functional on $\L_{q,l}$.

Further on, we will need to consider simultaneously the distribution $F$, that serves as a symbol of the Toeplitz operator, and the distribution $\Fb=\o F$ that enters in the expression for the sesquilinear form. We always denote them by the same letter, however they are distinguished by the font: the latter distribution is boldfaced.

So, suppose that a symbol $F\in\Dc'$ satisfies the condition $\Fb=\o F\in \Dc_q'$ for some $q>0.$
We define the \emph{Toeplitz sesquilinear form}, similar to \eqref{2.DistrE}, as
\begin{equation}\label{2.DistrD}
    \tb_F(u,v)=\langle \o F, u\bar{v}\rangle.
\end{equation}

If $u,v\in \Bc\cap\DF_\cF$, $\cF<\frac12$, then the product $u\bar{v}$ belongs to $\Dc_q$, $q<2\cF$, and, therefore,
the sesquilinear form $\tb_F(u,v)$ is well defined by \eqref{2.DistrD}.  $\o F\in \L_{q,l}'$,
with the estimate
\begin{equation*}\label{2.Distr.DE}
    |\tb_F(u,v)|\le C\|\o u\bar{v}\|_{\L_{q,l}}\le C\sum_{|\a|\le l}|D^{(\a)}u(z)|\sum_{|\a|\le l}|D^{(\a)}v(z)|e^{-(1-2\cF)|z|^2}.
\end{equation*}

\subsection{Boundedness}
It is well known that the functional of taking the value of the function at a given point is a continuous functional in Bergman spaces.  This property can be expressed  by  saying that  the delta-distribution belongs to the dual  of the Bergman space, under the natural $L^2$ - induced duality.  The same property, with the same easy proof using the Cauchy formula, holds for any distribution with compact support. In the present paper we consider the case of distributions without the condition of compact support imposed.
\begin{proposition}\label{4.boundedness.prop}Let $\Fb$ be a distribution in the class $\L_{q,l}'$ for some $q>1$, $F=\Fb\o^{-1}$. Then the Toeplitz  operator $\Tb(F)$ is bounded in $\Bc:$ for all $u,v\in \Bc,$
\begin{equation}\label{4.bound.1}
  |\tb_{F}(u,v)|= | \langle \Fb ,u\bar{v}\rangle|\le C(\Fb)\|u\|_{\Bc}\|v\|_{\Bc}.
\end{equation}
\end{proposition}
\begin{proof}
 By the definition of the class $\L_{q,l}'$ and the norm in $\L_{q,l}$,
\begin{equation}\label{4.bound.2}
    \langle \Fb ,u\bar{v}\rangle\le C\|u \bar{v}\|_{\L_{q,l}}\le C\|u\|_{\L_{q/2,l}}\|v\|_{\L_{q/2,l}}.
\end{equation}
Therefore, the estimate \eqref{4.bound.1} will follow from \eqref{4.bound.2} as soon as we prove the inequality
\begin{equation*}\label{4.bound.3}
    \|u\|_{\L_{q/2,l}}^2\le C \|u\|_{\Bc}^2.
\end{equation*}
For $l=0$, this inequality  is a particular case of Corollary 2.8 in \cite{Zhu}. The case of a positive $l$ is reduced to this one by using the inequality $|u^{(\a)}(z)|\le C_{\a}\int_{|z-\z|\le1}|u(\z)|d\l(\z)$, which follows from the Cauchy formula.
\end{proof}

\section{Finite rank operators and forms}
\subsection{Definitions}We start by recalling that an everywhere defined operator $\Tb$ in the Hilbert space $\Kc$ is called finite rank if for some elements $f_j,g_j\in\Kc, j=1,\dots,N$
\begin{equation}\label{2.FRbounded}
    \Tb u=\sum_{j=1}^N (u,g_j)f_j
\end{equation}
for all $u\in\Kc.$ As usual, it is much more convenient to use the sesquilinear form in the study of the properties of operators. In the language of sesquilinear forms, equivalently,
\begin{equation}\label{2.FRforms}
    (\Tb u,v)=\sum_{j=1}^N (f_j,v)(u,g_j)
\end{equation}
for all $u,v\in\Kc.$ The smallest number $N$ in such representations is called the rank of the operator. For uniformity, we say that the zero operator and only it has rank 0, i.e., the sum on the right in \eqref{2.FRbounded}, \eqref{2.FRforms} is empty.
By \eqref{2.FRbounded}, \eqref{2.FRforms}, a finite rank operator is automatically bounded.

We will consider a more general case, when the relation of the type \eqref{2.FRforms} holds not for all $u,v\in \Kc$ but only for $u,v$ in a certain linear subset $\Kc^0\subset\Kc$. If, still, $f_j,g_j\in\Kc$ and $\Kc^0$ is dense in $\Kc$, these two definitions are equivalent, by  continuity. We, however, are interested in the situation where the representation \eqref{2.FRforms} holds with $f_j,g_j\not\in\Kc$.

We denote by $\Bc^\circ$ the space of entire analytical functions, belonging to  $\DF_{1,-}$, $\Bc^\circ=\Ac\cap \DF_{1,-}$. For $f\in\Bc^\circ$ and $v$ of exponential growth, the expression
 $(f,v)$ is still correctly defined, although $f$ is not necessarily in $\Bc:$
\begin{equation*}\label{2.(fv)}
    (f,v)=\int_{\C}f(w)\o(w)\overline{v(w)}d\l(w),\ (v,f)=\overline{(f,v)},
\end{equation*}
and this definition is consistent with the definition of the scalar product in the space $\Bc$. In particular, $(f,v)$ is defined for $v$ being an analytical polynomial or the reproducing kernel. By continuity,
the reproducing relation \eqref{2.repro} extends to all $f\in \Bc^\circ:$
\begin{equation*}\label{2.repr.ext}
    (f, \k_z)=f(z).
\end{equation*}
Note also that the latter equation admits differentiation in $z$, since the
derivative of $\k_z$ is, again, of exponential growth: $\partial_{\bar{z}}^\a\k_z(w)=i^\a w^\a\k_z(w)$.

Now we can give a definition of more general finite rank operators and forms.
\begin{definition}\label{2.FRdefGen}Let $F$ be a function in $\DF_{1,-}$ or a distribution such that $\Fb=\o F\in \Dc_q'$ for some $q>0.$ We say that the sesquilinear  form $\tb=\tb_F,$ defined in  \eqref{2.sesqForm}, resp. \eqref{2.DistrD}, \emph{has finite rank on reproducing kernels} if, with some functions $f_j,g_j\in \Bc^\circ$, $j=1,\dots, N$
\begin{equation}\label{2.repr.gener}
  \tb(u,v)=\sum_{j=1}^N(u,{g_j})(f_j,v),
\end{equation}
for all $u,v$ being reproducing kernels, i,e., $u=\k_z, v=\k_{z'}.$ In other words,
\begin{equation*}\label{2.repr.frank}
    \tb(\k_z,\k_{z'})=\sum_{j=1}^N(\k_z,{g_j})(f_j,\k_{z'})
\end{equation*}
\end{definition}
For $f_j,g_j\in \Bc$, this definition is consistent with \eqref{2.FRforms}. However, for $f_j,g_j$ outside $\Bc$, the functionals on the right hand side in \eqref{2.repr.gener} are not continuous with respect to $u,v$ in the space $\Bc$, therefore, a sesquilinear form $\tb$ is not necessarily \emph{a priori} bounded in $\Bc$. Such boundedness will only follow \emph{post factum} from the finite rank theorems of this paper.

In a similar way, we say that the sesquilinear form has finite rank on polynomials, if for some functions $f_j,g_j\in \Bc^\circ$
\begin{equation*}\label{2.repr.polyn}
    \tb(w^k,w^{k'})=\sum_{j=1}^N(w^{k},g_j)(f_j,w^{k'})
\end{equation*}
for all $k,k'\in Z_+.$

It is easy to see that these two properties are equivalent. In one direction it follows form the relation $\bar{w}^k=\partial_z^k \k(z,w)_{z=0}$, in the other direction, it follows from the Taylor expansion for  $\k(z,w)$.
Further on, we will systematically use this equivalence.

The starting point of our analysis is the finite rank theorem established in \cite{BorRoz} (see Theorem 3.1 there):
\begin{theorem}\label{TheoremBorRoz} Suppose that the symbol $F$ is a function in $\DF_{1,-}$. If the sesquilinear form $\tb_F$ has finite rank on reproducing kernels, then $F=0.$\end{theorem}

\subsection{Infinite matrices. The bounded case}
The finite rank property of the sesquilinear forms is closely related with the properties of infinite matrices. With a distribution $F$ we associate two types of such matrices. For an infinite system of points $z_j\in \C$, we consider the matrix $\KF=\KF(F)$ with elements $\kF_{k,k'}=\tb_F(\k_{z_k},\k_{z_k'})$.
Another infinite matrix, $\PF=\PF(F)$, associated with $F$, is defined by setting $\pF_{k,k'}=\tb({w^k,w^{k'}})=\langle \o F, w^k,\bar{w}^{k'}\rangle$. If the sesquilinear form  $\tb_F$  has finite rank on polynomials and on reproducing kernels (with all possible collections of points $z_j$), then the rank of  $\PF(F)$, equals $\sup_{\{z_j\}}\rank(\KF(F))$, with $\sup$ taken over all collections of points $z_j\in\C.$

Under proper conditions imposed on the symbol, the converse statement is correct as well.

\begin{proposition}\label{3.converseFRProp} Let the sesquilinear form $\tb_F(u,v)$ defined in \eqref{2.DistrD}, with a distribution $F\in\Ec'(\R)$ be bounded in $\Bc:$
\begin{equation*}\label{ConvProp.1}
    |\langle F, \o u\bar{v}\rangle|\le C \|u\|_{\Bc}\|v\|_{\Bc}, \ \ u,v\in\Bc.
\end{equation*}
Suppose that the infinite matrices $\PF(F)$, $\KF(F)$ have finite rank, not greater than $N$. Then the operator $\Tb(F)$ has finite rank not greater than $N$ in the sense of \eqref{2.FRforms}.
\end{proposition}
\begin{proof} The proof will be given for polynomials; for reproducing kernels, it follows from the equivalence explained earlier. We start by showing that the range of the operator $\Tb=\Tb(F)$ has finite dimension, $\dim{\Ran \Tb}\le N$. To do this, we suppose that, in the opposite,  $\dim{\Ran \Tb}> N$. This means that there exist at least $N+1$ functions $u_j\in \Bc,$ $j=1,\dots,N+1,$ such that the functions $\Tb u_j$ are linearly independent.
Since polynomials are dense in the space $\Bc$, for any $j$ there exists a  sequence of polynomials $p_{j,n}$, $n=1,2,\dots,$ such that $p_{j,n}$ converges to $u_j$ in $\Bc$ as $n\to\infty$. By continuity, this implies that for $n$ large enough, the system of functions $q_{j,n}=\Tb p_{j,n},$ $j=1,\dots, N+1,$ is linearly independent. We fix such, sufficiently large, $n$ and will omit it in notations further on, writing $p_j=p_{j,n}$, $q_j=q_{j,n}$.  The Gram matrix $\GF$ of $N+1$ linearly independent functions $q_j$, i.e.,  the matrix with elements $\gF_{jk}=(q_j,q_k)$, has maximal rank, $\rank(\GF)=N+1$. Therefore, repeating the polynomial approximation procedure, approximating the functions $q_j$ by polynomials $r_k$, we obtain that the matrix $\HF$ with elements $(q_j, r_k)$, $j,k=1,\dots,N+1$ has rank $N+1$ for some polynomials $r_k, k=1,\dots,N+1$. Finally, we recall that $(q_j, r_k)=(\Tb p_j,r_k)$, i.e.,  it is the value of the sesquilinear form of the operator $\Tb$ computed on polynomials $p_j,r_k$ . Therefore,  $\HF$ is a $(N+1)\times (N+1)$ sub-matrix of the matrix obtained by linear operations with columns and rows from the matrix $\PF(F)$. Such operations cannot increase the rank of the matrix, so $\rank(\HF)\le N$, which contradicts the previously obtained equality $\rank(\HF)=N+1$. This contradiction shows that, in fact, $\dim{\Ran \Tb}\le N$.

Now, to prove that the operator $\Tb$ has finite rank, i.e., that the representation  \eqref{2.FRbounded} holds,
we take as the system $f_j$ a linearly independent \emph{orthonormal}  system of functions in the range of $\Tb$. Thus, for any $u\in \Bc$, we have
\begin{equation*}\label{ConvProp.2}
    \Tb u=\sum_{j=1}^N(\Tb u,f_j)f_j=\sum_{j=1}^N( u,g_j)f_j, \ \ g_j=\Tb^*f_j.
\end{equation*}
\end{proof}
\subsection{Scaling.}Now we dispose of the condition of the boundedness of the operator $\Tb$, which was quite instrumental in the above proof of Proposition \ref{3.converseFRProp}.

In the course of this study we will be using the scaling operator $S_t$ (cf. \cite{BauLe}, where this operator was used for the analysis of Toeplitz operators with symbols-functions.)

For a function $\psi$ on $\C^1$, we set $S_t\psi(z)=t\psi(tz)$, $t>0$.
For a distribution $\Fb$ we define the distribution  $W_t \Fb$ by setting
\begin{equation}\label{4.scaling}
\langle W_t\Fb,\psi\rangle=\langle \Fb, S_{t^{-1}} \psi\rangle,
\end{equation}
and
$F_t=e^{(1-t^2)|z|^2}W_t (\o(z)^{-1}\Fb).$

\begin{proposition}\label{prop.scaling} For $F\in \L_{q,l}', q>0,$ the equality holds
\begin{equation}\label{scaling5}\tb_{F_t}(u,v)=t^{-1}\tb_{F}(S_{t^{-1}} u,S_{t^{-1}} \bar{v}).\end{equation}
\end{proposition}
\begin{proof} We set $\Fb=\o F$, thus $W_t \Fb= \o(z)F_t$. By the definition of the sesquilinear  form $\tb_F$ and the transformation $W_t$,  we have, for $u,v\in \Bc\cap\DF_{\frac12}$:
\begin{gather*}\label{4.scaling.2}
    \tb_{F_t}(u,v)=\langle F_t, \o u \bar{v}\rangle = \langle W_t \Fb, u \bar{v}\rangle=\\ \nonumber\langle \Fb, S_{t^{-1}}(u \bar{v})\rangle=t^{-1}\langle \Fb, S_{t^{-1}}u S_{t^{-1}}\bar{v})\rangle.
\end{gather*}
\end{proof}
\begin{proposition}\label{3.Scaling}Let $\Fb\in \L_{q,l}', q>0$. Then
\begin{equation}\label{3.Scaling.1}
    W_t\Fb\in \L_{q',l}'
\end{equation}
for any $q'< t^2q$.
\end{proposition}
\begin{proof} A direct calculation shows that for $\psi\in \L_{q,l}$, the function $S_{t^{-1}} \psi$ belongs to $\L_{t^2q,l}$, with the corresponding  norm estimate. Therefore, \eqref{3.Scaling.1} follows immediately from the definition \eqref{4.scaling} of the transformation $W_t$ and Propositions , \ref{prop.scaling}.
\end{proof}

\begin{proposition}\label{PropBdd} Suppose that $\Fb\in \L_{q,l}'$ with $q>0$. Then, for sufficiently large $t$, the form $\tb_{F_t}$ is bounded in $\Bc$:
\begin{equation*}\label{4.scaling1}
    |\tb_{F_t}(u,v)|=|\langle W_t\Fb, u\bar{v}\rangle|\le C \|u\|_{\Bc}\|v\|_{\Bc}.
\end{equation*}
\end{proposition}
\begin{proof} By choosing a sufficiently large $t$, we can make the number $t^2q$ larger than $2$, and then the boundedness, due to the relation \eqref{scaling5}, follows from Proposition \ref{4.boundedness.prop}.
\end{proof}

\begin{proposition}\label{3.prop.invar}Suppose that for a distribution $F\in \DF_{\cF}$, the infinite matrices $\PF(F), \KF(F)$
have finite rank. Then for $t>1$, the infinite matrices $\PF(F_t), \KF(F_t)$
have the same finite rank.
\end{proposition}
\begin{proof}By \eqref{scaling5}, the scaling leads to a simple transformation of the elements of $\PF:$ $\pF(F_t)_{k,k'}=t^{-k-k'}\pF(F)_{k,k'}$. So, each horizontal row and each column is multiplied by a constant. Such operations cannot change the rank. As for the matrix $\KF$, the elements of $\KF(F_t)$ are again the reproducing kernels, just calculated at different points: $\tb_{F_t}(\k_z,\k_{z'})=\tb_{F}(\k_{tz},\k_{tz'})$, and, again, the rank does not grow.
\end{proof}
Combining propositions \ref{3.converseFRProp}, \ref{PropBdd}, and \ref{3.prop.invar}, we arrive at the  result on the finite rank forms.
\begin{proposition}Suppose that $\Fb=\o F$ be a distribution in $\L_{q,l}'$ with some $q>0$, such that the infinite matrices $\PF(F)$, $\KF(F)$ have finite rank, not greater than $N$. Then for sufficiently large $t$, the Toeplitz operator $\Tb(F_t)$ is bounded and has finite rank not greater than $N$.
\end{proposition}



\section{The $\dbar$ equation for \emph{functions} in $\Dc_q'$}\label{4.dbar functions}
The rest of the paper is devoted to proving the version of Theorem \ref{TheoremBorRoz}  for $F$ being a distribution with $\Fb=\o F\in \Dc_q', q>0.$ Note that this condition allows  a rather rapid growth of $F$ at infinity. The proof of this theorem, given in Section 3 there, being applied to distributions, goes through smoothly, up to the point where the decay of the two-dimensional Fourier transform of $\o F$ is used. For distributions, this decay property does not hold, and the proof breaks down.

The only presently existing proof of the finite rank theorem for distributions, see \cite{AlexRoz}, uses the reduction of  a given  finite rank Toeplitz operator to some other Toeplitz operator, also finite rank, but now with symbol-function. The critical feature  here is an elementary property of the $\dbar$-equation in the class of compactly supported distributions, see Lemma 3.2 in \cite{AlexRoz} or Lemma 1.2 and the discussion in the Introduction.

The aim of this section  and the next one is to establish a similar property for distributions, not necessarily  having compact support but decaying at infinity in the sense of Section 2.
We are going to prove the following statement.
\begin{theorem}\label{4.th.dbar} Let $\hb$ be a distribution in $\Dc_q'$, with some $q>0$. Then the following two properties are equivalent:\\
(a) for any $q'\in(0,q.e)$, there exists a distribution $\gb\in \Dc_{q'}'$,, such that
\begin{equation*}\label{4.thm.1}
    \dbar \gb=\hb;
\end{equation*}
(b) the equality
\begin{equation*}\label{4.thm.2}
    \langle \hb,z^k\rangle=0
\end{equation*}
holds for all $k\in\Z_+.$
\end{theorem}
The implication (b)$\Rightarrow$(a) is obvious.
The proof of the inverse implication will consist of several steps. First, we establish the property in question  for functions. In Section 5 the result will be extended to distributions.

Let $\hb(z)$ be a \emph{function} in $\L_{q,l}'$. Recall that this means that for any $\psi\in\L_{q,l}$, the inequality holds
\begin{equation}\label{4.Th1.1}
    \langle \hb,\psi\rangle=\int_{\C}\hb(z)\vf(z)d\l(z)\le C(\hb)\|\psi\|_{\L_{q,l}}.
\end{equation}
\begin{theorem}\label{4.Th.dbar.func}Suppose that a \emph{function} $\hb$ satisfies \eqref{4.Th1.1} for some $q>0, l\ge0$ and is orthogonal, as a distribution, to all analytic polynomials, $\langle \hb,z^k\rangle=0$, $k=0,1,\dots$. Then, for any  $q_1\in(0,q/e)$, there exists a function $\ub\in  \L_{q_1,l}'$ such that $\dbar \ub=\hb$, moreover,
\begin{equation*}\label{4.Th.dbar.1a}
 \|\ub\|_{\L_{q_1,l}'}\le C \|\hb\|_{\L_{q,l}'},
\end{equation*}
 or, equivalently,
\begin{equation}\label{4.Th.dbar.1b}
   |\langle \ub,\psi\rangle|\le C \|\hb\|_{\L_{q,l}'}\|\psi\|_{\L_{q_1,l}}
\end{equation}
for all $\psi\in\L_{q_1,l}$.
\end{theorem}
\begin{proof} We set
\begin{equation}\label{4.Th.dbar.1}
\ub(z)=-\frac{1}{\pi}\int_{\C}\frac{\hb(\z)}{\z-z}d\l(\z)
\end{equation}
and will prove \eqref{4.Th.dbar.1b} for $q_1<q/e$.

The integral in \eqref{4.Th.dbar.1}, obviously, converges, and, since $(-\pi z)^{-1}$ is the fundamental solution for $\dbar$, the function $\ub$ satisfies the equation $\dbar \ub=\hb$. What, actually, we need to establish, is inequality \eqref{4.Th.dbar.1b}, in other words, that the function $\ub$ satisfies the decay conditions required by $\ub\in  \L_{q_1,l}'$. We will be proving  inequality \eqref{4.Th.dbar.1b} with $q_1<q/e$ for $\psi\in C_0^\infty(\C)$; with the constant in \eqref{4.Th.dbar.1b} not depending on $\psi$; since  $C_0^\infty(\C)$ is dense in  $\L_{q_1,l}$, this estimate extends to the whole of $\L_{q_1,l}$ by continuity.

For a  $\g>0$, depending on $z$, to be determined later, we introduce the functions $\th_\g, \vs_\g\in C^{\infty}(\R^1_+)$ in the following way:
the function $\th_\g(t)$, whith  values between $0$ and $1$, equals $0$ for $t\le \g$, equals $1$ for $t\ge \g+1$, and the function $\vs_\g(t)=1-\th_\g(t)$.

For $\psi\in C_0^\infty(\C)$, we can write
\begin{gather*}\label{4.Th.dbar.2}
    \langle \ub,\psi \rangle=-\frac1\pi \int_\C \int_\C\frac{(\th_{\g(z)}(\z)+\vs_{\g(z)}(\z))\hb(\z)\psi(z)}{\z-z}d\l(z)d\l(\z)=\\\nonumber
    -\frac{1}{\pi}\int_{\C} \hb(\z)\left(\int_\C\frac{\th_{\g(z)}(\z)\psi(z)}{\z-z} d\l(z)\right)d\l(\z)\\ \nonumber-\frac1\pi\int_{\C} \psi(z)\left(\int_{\C}\frac{\vs_{\g(z)}(\z)\hb(\z)}{\z-z}d\l(\z)\right)d\l(z)=-\frac1\pi I_1-\frac1\pi I_2,
\end{gather*}
where $\g=\g(z)$ is chosen as $\frac{|z|}{\sqrt{e}}$.

We start with studying $I_1$. Here, the integration is performed over the region $|z|\le |\z| \sqrt{e}$. We introduce the function
\begin{equation*}
    H(\z)=\int_{\C}\frac{\th_{\g(z)}(\z)\psi(z)}{\z-z}d\l(z).
\end{equation*}
One can understand  $I_1$ as the action of the distribution $\hb$ on the function $H$. So, in order to estimate $I_1$, we need estimates for the norm of the function $H$ in $\L_{q,l}.$

Passing to absolute values, we obtain
\begin{equation}\label{4.Th.dbar.3}
    |H(\z)|\le \max_{|z|\le |\z|\sqrt{e}}|\psi(z)|\int_{|z|\le |\z|\sqrt{e}}\frac{d\l(\z)}{|\z-z|}\le C|\z|\max_{|z|\le |\z|\sqrt{e}}|\psi(z)|.
\end{equation}

We need also similar estimates for derivatives of $H(\z)$, up to order $l$. We show here such estimate for $\partial H/\partial \z_1$, other derivatives are estimated in a similar way.

We write
\begin{gather*}\label{5.d1}
    H(\z)=H_1(\z)-H_2(\z)=\\ \nonumber
   = \int_{\C}\frac{\psi(z)}{\z-z}d\l(z)-\int_{\C}\frac{(1-\th_{\g(z)}(\z))\psi(z)}{\z-z}d\l(z).
\end{gather*}
For the derivative of $H_1$, we have
\begin{equation*}
    \frac{\partial}{\partial \z_1} H_1=\int_{\C}\frac{\frac{\partial}{\partial z_1} \psi(z)}{\z-z}d\l(z),
\end{equation*}
since the derivative and the convolution commute.

In $H_2, $ the integrand does not have singularities, so we can differentiate under the integral sign:
\begin{equation*}
    \frac{\partial H_2}{\partial \z_1}=\int_\C\frac{\partial \th_{\g(z)}(\z)}{\partial\z_1}\frac{\psi(z)}{\z-z}d\l(z)+
    \int_{\C}(1-\th_{\g(z)(\z)})\frac{\partial}{\partial \z_1}\frac{\psi(z)}{\z-z}d\l(z).
\end{equation*}
After partial integration in the second integral, using $\frac{\partial}{\partial \z_1}(\z-z)^{-1}=-\frac{\partial}{\partial z_1}(\z-z)^{-1}$,
we obtain

\begin{gather*}
     \frac{\partial H_2}{\partial \z_1}=\int_\C\frac{\partial \th_{\g(z)}(\z)}{\partial\z_1}\frac{\psi(z)}{\z-z}d\l(z)+\\ \nonumber
    \int_{\C}(1-\th_{\g(z)(\z)})\frac{\partial\psi(z)}{\partial z_1}\frac{d\l(z)}{\z-z}+\int_{\C}\frac{(\partial/\partial_{\z_1}+\partial/\partial_{z_1})\th_{\g(z)}(\z) \psi(z)}{\z-z}d\l(z).
\end{gather*}
Collecting the expressions for the derivatives, we arrive at
\begin{equation*}
   \frac{\partial H}{\partial \z_1}=\int_{\C}\th_{\g(z)}(\z)\frac{\partial \psi(z)}{\partial z_1}\frac{d\l(z)}{\z-z}+\int_{\C}\frac{(\partial/\partial_{\z_1}+\partial/\partial_{z_1})\th_{\g(z)}(\z) \psi(z)}{\z-z}d\l(z).
\end{equation*}
the first integral is estimated via the bound of the derivative of $\psi(z)$ and the second one via the bound for $\psi(z)$, as in \eqref{4.Th.dbar.3}. We can repeat this reasoning for higher derivatives of $H(\z)$, which leads to the estimate
\begin{gather*}\label{4.Th.dbar.4}\nonumber
    |D^{\a}H(\z)|\le C \max_{|z|\le C|\z|\sqrt{e},|\b|\le|\a|}|D^{\b}\psi(z)|\int_{|z|\le |\z|\sqrt{e}}|\z-z|^{-1}d\l(z)\le \\  C |\z|\max_{|z|\le C|\z|\sqrt{e}}\max_{|\b|\le|\a|}|D^{\b}\psi(z)|.
\end{gather*}

Now we can estimate the norm of $H(\z)$ in the class $\L_{q,l}$:
\begin{gather*}\label{4.Th.dbar.5}
    e^{-q|\z|^2}\max_{|\a|\le l}|D^{\a}H(\z)|\le C e^{-q|\z|}|\z|\max_{|\a|\le l}\max_{|z|\le \sqrt{e}|\z|}|D^{\a}\psi(z)|\le\\ \nonumber C e^{-q_1(|\z|\sqrt{e})^2}\max_{|\a|\le l}\max_{|z|\le \sqrt{e}|\z|}|D^{\a}\psi(z)|,
    \end{gather*}
with $q_1<q/e$. So, the function $H$ belongs to $\L_{q,l}$,
$\|H\|_{\L_{q,l}}\le C \|\psi\|_{\L_{q_1,l}},$    and, by our assumptions about $\hb$,
\begin{equation*}\label{4.Th.dbar.6}
    |I_1|=|
    \langle \hb, H\rangle|\le C \|\psi\|_{q_1,N}.
\end{equation*}

We pass to estimating $I_2$. Here we will need the orthogonality condition.

For a given $n$, we set $r_n=\sqrt{\frac{ne}{2q}}$. Let $\vartheta_n$, $n=1,\dots$, be a partition of the unit, $\sum\vartheta_n=1$, such that $\vartheta_1\in C_0^\infty[0,r_2)$, $\vartheta_n\in C^\infty_0(r_{n-1},r_{n+2})$, for $n\ge2$.

So, we have
\begin{gather*}\label{4.Th.dbar.7}
    I_2=\int_{\C}\sum_{n=1}^{\infty}\vt_n(|z|)\psi(z)\int_{\C}\vs_{\g(z)}(\z)\hb(\z)\frac{d\l(\z)}{\z-z}d\l(z)=\\ \nonumber \sum_{n=1}^{\infty}\int_\C\vt_{n}(z)\psi(z)\Theta(z)d\l(z),
\end{gather*}
where
\begin{equation*}
    \Theta(z)=\int_{\C} \vs_{\g(z)}(\z)\hb(\z)\frac{d\l(\z)}{\z-z}.
\end{equation*}
We replace  the fraction $\frac{1}{\z-z}$  in this integral by its expansion:
\begin{gather}\label{4.Th.dbar.8}
    I_2=\sum_{n=1}^{\infty}\int_{\C}\vt_n(z)\psi(z)\times\\ \nonumber\int_{\C}\left(-\frac1z-\frac{\z}{z^2}
    -\dots-\frac{\z^{m_n}}{z^{m_n+1}}-
    \frac{\z^{m_n+1}}{z^{m_n+2}}\frac{1}{\z-z}\right) \hb(\z)\vs_{\g(z)}(\z)d\l(\z)d\l(z),
\end{gather}
where we set $m_n=n$ for $n\le 3$ and $m_n=n-2$ otherwise.
Now we use the \emph{orthogonality conditions}, which give
\begin{equation*}\label{4.Th.dbar.9}
    -\int \z^k \hb(\z)\vs_{\g(z)}(\z)d\l(\z)=\int_{\C}\z^k\hb(\z)\th_{\g(z)}(\z)d\l(\z).
\end{equation*}
Therefore, the term $I_2$ equals
\begin{gather}\label{4.Th.dbar.10}
\sum_{n=1}^{\infty}\int_{\C}\vt_n(z)\psi(z)\left(\int_{|\z|>\frac{|z|}{\sqrt{e}}}
    \left(\frac1z+\frac{\z}{z^2}+\dots+\frac{\z^{m_n}}{z^{m_n+1}}\right)\hb(\z)
    \th_{\g(z)}(\z)d\l(\z)\right)d\l(z)\\ \nonumber
    -\sum_{n=1}^{\infty}\int_{\C}\vt_n(z)\psi(z)\left( \int_{|\z|<\frac{|z|}{\sqrt{e}}+1}\frac{\z^{m_n+1}}{z^{m_n+2}}\frac{\hb(\z)}{\z-z} \vs_{\g(z)}(\z)\l(\z)\right)d\l(z)=I_2'-I_2''.
\end{gather}
We consider the terms  in $I_2'$ with $n\ge4$ in detail; small values of $n$ are treated similarly, with minor changes. In the first line in \eqref{4.Th.dbar.10}, we have, for $k\le n-2$, the  terms of the form
\begin{equation}\label{4.Th.dbar.11}
    \int_{\C}\vt_n(z)\psi(z){z^{-k-1}}
    \left(\int_{|\z|>\frac{|z|}{\sqrt{e}}}\th_{\g(z)}(\z)\z^{k}\hb(\z)d\l(\z)\right)d\l(z).
\end{equation}
We denote by $\vb_k(z) $ the inner integral in \eqref{4.Th.dbar.11}.
It can be represented as
\begin{equation}\label{4.Th.dbar.11a}
    \vb_k(z)=\langle \hb(\z), \th_{\g(z)}(\z)\z^k\rangle.
\end{equation}
An elementary calculus shows that, for $|\z|>|z|/\sqrt{e}$,
\begin{equation*}\label{4.Th.dbar.12}
    \max_{|\z|>\frac{|z|}{\sqrt{e}}}e^{-q|\z|^2}|\z|^k\le Ce^{-\frac{q}{e}|z|^2}\frac{|z|^k}{e^{k/2}},
\end{equation*}
therefore,
\begin{equation}\label{4.Th.dbar.13}
|\th_{\g(z)}(\z)\z^k|\le Ce^{q|\z|^2}e^{-\frac{q}{e}|z|^2}\frac{|z|^k}{e^{k/2}}.
\end{equation}
Similarly, for derivatives of the expression \eqref{4.Th.dbar.11a}, we have estimates of the same kind,
\begin{equation}\label{4.Th.dbar.14}
   |D^{\a}(\th_{\g(z)}(\z)\z^k)| \le C_{\a}e^{q|\z|^2}e^{-\frac{q}{e}|z|^2}\frac{|z|^k}{e^{k/2}}.
\end{equation}
So,  by \eqref{4.Th.dbar.13}, \eqref{4.Th.dbar.14}, the right-hand side in these inequalities gives an estimate for the norm of the function $\th_{\g(z)}(\z)\z^k$ in the space $\L_{q,l}$
Therefore, since $\hb\in \L_{q,l}'$, for the function $\vb_k(z)$ defined in \eqref{4.Th.dbar.11a} we have, by \eqref{4.Th1.1},
\begin{equation}\label{4.Th.dbar.15}
    |\vb_k(z)|=|\langle \hb(\z), \th_{\g(z)}(\z)\z^k\rangle|\le C e^{-\frac{q}{e}|z|^2}\frac{|z|^k}{e^{k/2}}.
\end{equation}
For the corresponding term in \eqref{4.Th.dbar.11}, we obtain from \eqref{4.Th.dbar.15}:
\begin{gather*}
    \left|\int_{\C}\vt_n(z)\psi(z)z^{-k-1}\vb_k(z) d\l(z)\right|\le C \int_{\C}\vt_n(z)|\psi(z)||z|^{-k-1}\frac{|z|^k}{e^{\frac{k}{2}}}d\l(z),
\end{gather*}
and, after the summation,
\begin{gather*}
    \left|\int_\C\sum_{k=0}^{n-2}\vt_n(z)\psi(z)z^{-k-1}\vb_k(z)d\l(z)\right|\le \\ \nonumber
    C\sum_{k=0}^{n-2}\int_\C \vt_n(z)|\psi(z)||z|^{-k-1}e^{-\frac{q}{e}|z|^2}|z|^k e^{-\frac{k}{2}}d\l(z)\le\\ \nonumber
    C\int_{r_{n-1}\le|z|\le r_{n+2}}|\psi(z)|e^{-\frac{q}{e}|z|^2}d\l(z),
\end{gather*}we have majoration by
\begin{gather}\label{4.Th.dbar.16}
   |I_2'|\le
   C \sum_{n=1}^\infty \left(\max_{t\in(r_{n-1},r_{n+2})}\max_{|z|<t, \a+\b\le N }|D^{\a,\b}\psi(z)| e^{-q_1t^2}\right)\int_{\C}e^{(q_1-\frac{q}{e})|z|^2}d\l(z).
\end{gather}
Estimate \eqref{4.Th.dbar.16} takes care of the  term $I_2'$ in \eqref{4.Th.dbar.10}. Now we study the remainder, $I_2''$. Again, we consider the terms with $n\ge4$ (the case of small $n$ is even simpler). So, we study the integral
\begin{gather}\label{4.Th.dbar.17}
  \sum\int_{\C}\vt_n(z)\psi(z)\left(\int_{|\z|\le \frac{|z|}{\sqrt{e}}+1} \vs_{\g(z)}(\z)\frac{\z^{n-1}}{z^n}\frac{\hb(\z)}{\z-z}d\l(\z)\right)d\l(z)=\\ \nonumber
  \sum\int_{\C}\vt_n(z)\psi(z)z^{-n}\left(\int_{|\z|\le\frac{|z|}{\sqrt{e}}+1}
  \vs_{\g(z)}(\z)\frac{\z^{n-1}}{\z-z}\hb(\z)d\l(\z)\right)d\l(z).
\end{gather}
For the inner integral in \eqref{4.Th.dbar.17}, we have the representation in the form of the action of $\hb$, considered as distribution, on the given functions:
\begin{gather}\label{4.Th.dbar.18}
\left|\int_{|\z|\le\frac{|z|}{\sqrt{e}}+1}
  \vs_{\g(z)}(\z)\frac{\z^{n-1}}{\z-z}\hb(\z)d\l(\z)\right|=
  |\langle \hb(z), \vs_{\g(z)}(\z)\frac{\z^{n-1}}{\z-z}\rangle|\le \\ \nonumber C\|\hb\|_{\L_{q,l}'}\max_{|\z|\ge |z|/\sqrt{e}}\max_{\a_1+\a_2\le l}\{|D^{\a_1,\a_2}(\vs_{\g(z)}(\z)\frac{\z^{n-1}}{\z-z})|e^{-q|\z|^2}\}.
\end{gather}
The absolute value of the  denominator ${\z-z}$ in \eqref{4.Th.dbar.18} is bounded below by $(1-e^{-1/2})|z|$, so the derivatives in \eqref{4.Th.dbar.18} can be upper  bounded by  $C|z|^{n-1}e^{(n-1)/2}$, just like this was done in \eqref{4.Th.dbar.14}, \eqref{4.Th.dbar.15}:
\begin{equation*}\label{4.Th.dbar.19}
    \left| z^{-n}\int_{\C}\vs_{\g(z)}(\z)\frac{\z^{n-1}}{\z-z}\hb(\z)d\l(\z)\right|\le C\|\hb\|_{\L_{q,l}'}|z|^{-n}\frac{|z|^{n-1}}{e^{\frac{n}2}}.
\end{equation*}
We substitute this estimate into \eqref{4.Th.dbar.17}, to obtain
\begin{gather}\label{4.Th.dbar.20}
    \left|\int_{\C}\vt_n(z)\psi(z)z^{-n}\int_{\C}\vs_{\g(z)}(\z) \frac{\z^{n-1}}{\z-z}\hb(\z)d\l(\z)d\l(z)\right|\\ \nonumber
    \le C\|\hb\|_{\L_{q,l}'}\int \vt_n(z)|\psi(z)|e^{-n/2}d\l(z) \\ \nonumber \le C\|\hb\|_{\L_{q,l}'}\int_{|z|\in(r_{n-1},r_{n+2})}|\psi(z)|e^{-n/2}d\l(z).
\end{gather}
Now we note that for $|z|\in[r_{n-1}, r_{n+2}]$,
\begin{equation*}
    |z|^2=r_n^2+O(1)=\frac{e}{2q}n+O(1),
\end{equation*}
and, therefore, $\frac{n}{2}=\frac{q}{e}|z|^2+O(1)$.
So, we can replace $e^{-n/2}$ by $e^{-\frac{q}{e}|z|^2}$ in the last integral in \eqref{4.Th.dbar.20} and, since the intervals $(r_{n-1},r_{n+2})$ form a covering of the real line with multiplicity less than 5, we can sum the inequalities of the form \eqref{4.Th.dbar.20} and arrive at
\begin{equation}\label{4.Th.dbar.21}
   |I_2'| \le C \|\hb\|_{\L_{q,l}'}\int|\psi(z)|e^{-\frac{q}{e}|z|^2}d\l(z)\le C\|\hb\|_{\L_{q,l}'}\sup_{t}\max_{|z|\le t}\{|\psi(z)|e^{-q_1t^2}\},
\end{equation}
and this concludes the proof of Theorem \ref{4.Th.dbar.func}.
Note, that we needed the bounds involving derivatives of $\psi$  only when estimating the term $I_1$, while the orthogonality condition was used only when estimating $I_2$.
\end{proof}

\section{$\dbar$ -estimates for distributions}\label{5.dbar Distr}
The aim of this section is to carry over the estimates of Sect.\ref{4.dbar functions} to distributions in $\L_{q,l}'$, thus finishing the proof of Theorem \ref{4.th.dbar}.

Let $\r(z)$ be a function in $C_0^\infty$, $\r(z)=\r(|z|)$,  $\r(z)=0$ for $|z|>1$ and $\int\r(z)d\l(z)=1$. For $\d>0$, we denote by $\r_\d(z)$ the function $\d^{-2}\r(\d^{-1}z)$. For any distribution in $\hb\in\Dc'(\C), $ (including functions), we set $\hb_\d=\hb*\r_\d$. Of course, $\hb_\d\in C^\infty$, and it is well known that $\hb_\d\to \hb$ in the sense of distributions in classes $\Dc', \Ec', \Sc'$. Our first aim is to establish similar convergence results in our classes $\L_{q,l}$ and $\L_{q,l}'$.

\begin{lemma}\label{5.lem1} Let $\psi$ be a function in $\L_{q,l}$, $q>0$.
Then, for any $q'>q$, $\psi_\d\to \psi$ in $\L_{q',l-1}$, moreover, this convergence in uniform in the following sense:  there exists a function $\tau(\d)$, $\tau(\d)\to 0$ as $\d\to\infty$, not depending on $\psi$, such that
\begin{equation*}
    \|\psi_\d-\psi\|_{\L_{q',l-1}}\le \tau(\d)\|\psi\|_{\L_{q,l}}.
\end{equation*}
\end{lemma}
\begin{proof}We consider the case $l=1$.  The general case follows by applying the same reasoning to derivatives of $\psi_\d$, since the derivative commutes with the mollification: $D^\a(\psi_\d)=(D^\a \psi)_\d$.

Consider the function
 \begin{equation*}
 \psi_\d(z)=\int \psi(z-\z)\r_\d(\z) d\l(\z)=\int \psi(z-\d \z)\r(\z)d\l(\z).
\end{equation*}
We subtract the quantity $\psi(z)=\int \psi(z)\r(\z)d\l(\z)$.
So,
\begin{equation}\label{5.lem1.1}
   | \psi_{\d}(z)-\psi(z)|=|\int(\psi(z-\d \z)-\psi(z)) \t(\z)d\l(\z)|\le \max_{|\z|\le1}|\psi(z-\d \z)-\psi(z)|.
\end{equation}
We can estimate the difference on the right-hand side in \eqref{5.lem1.1} via the derivative,
\begin{equation*}
    |\psi(z-\d \z)-\psi(z)|\le\d \max_{|z-\x|\le \d}(|D_{1}\psi(\x)|^2+|D_{2}\psi(\x)|^2)^{\frac12}.
\end{equation*}
By the definition of the class $\L_{q,1}$, we obtain now
 \begin{equation*}\label{5.lem1.2}
 |\psi(z-\d \z)-\psi(z)|\le \|\psi\|_{\L_{q,1}} \max_{|z-\z|\le\d}e^{q|\z|^2}\le C\d \|\psi\|_{\L_{q,1}}e^{q'|\z|^2},
 \end{equation*}
 and this inequality is exactly the statement of the lemma, with $\t(\d)=C\d$.
\end{proof}
Now we carry over this convergence result to distribution.
\begin{lemma}\label{5.lem2} Suppose that $\hb\in \L_{q',l}'$ is a distribution, with certain $q'>0$. Then for any $q\in(0,q')$, the distributions $\hb_\d=\hb*\r_\d$ converge to $\hb$ in $\L_{q,l+1}'$, uniformly in the sense that for a certain function $\t(\d)$, $\t(\d)\to 0$ as $\d\to 0$,
\begin{equation}\label{5.lem2.0}
    \|\hb-\hb_\d\|_{\L_{q,l+1}'}\le \t(\d)\|\hb\|_{\L_{q',l}'}.
\end{equation}
\end{lemma}
\begin{proof} By the standard definition of the norm in the dual Banach space, the required inequality \eqref{5.lem2.0} is equivalent to
\begin{equation*}\label{5.lem2.1}
    |\langle(\hb-\hb_d), \psi\rangle|\le \t(\d)\|\hb\|_{\L_{q',l}'}\|\psi\|_{\L_{q,l+1}},
\end{equation*}
for all $\psi\in \L_{q,l-1}$.  Again, we consider the leading case $l=1$.
We have
\begin{equation*}\label{5.lem2.2}
    \langle(\hb-\hb_d), \psi\rangle=\langle \hb,\psi\rangle-\langle \hb_d,\psi\rangle=\langle \hb,\psi\rangle-\langle \hb,\psi_\d\rangle=\langle \hb,(\psi-\psi_\d)\rangle.
\end{equation*}
Therefore, by our assumptions on $\hb$,
\begin{equation*}\label{5.lem2.3}
    |\langle(\hb-\hb_d), \psi\rangle|\le \|\hb\|_{\L_{q',l}'}\|\psi-\psi_d\|_{\L_{q',l+1}}.
\end{equation*}
Now, By Lemma \ref{5.lem1},  we can estimate the last expression, arriving at  \eqref{5.lem2.0}.
\end{proof}
We are now able to conclude the proof of Theorem \ref{4.th.dbar}.
Let $\hb$ be a distribution satisfying the conditions of the Theorem. Consider the mollified distributions $\hb_\d=\hb*\r_\d$. Since $\hb_\d$ is orthogonal to polynomials, by Theorem \ref{4.Th.dbar.func}, for any $\d>0$ there exist a solution $\ub_\d$ of the equation $\dbar \ub_\d=\hb_\d$. So, the difference $\ub_{\d}-\ub_{\d'}$ solves the equation
\begin{equation}\label{5.proof.1}
    \dbar(\ub_{\d}-\ub_{\d'})=(\hb_{\d}-\hb_{\d'}).
\end{equation}
The right-hand side in \eqref{5.proof.1} converges to zero in the norm of $\L'_{q,l}$ as $\d,\d'\to 0$, by Lemma \ref{5.lem2}. Therefore, by \eqref{4.Th.dbar.1a}, $\ub_{\d}-\ub_{\d'}$ converges to zero in the norm of $\L'_{q',l+1}$. This latter property implies that the family $\ub_{\d}$ converges to some distribution $\ub$ in the norm of $\L'_{q',l+1}$, and this distribution satisfies the equation $\dbar\ub=\hb$, which is proved by passing to the limit:

\begin{gather*}\label{5.proof.2}
\langle\hb,\psi\rangle=\lim_{\d\to0}\langle \hb_\d,\psi\rangle=\lim_{\d\to0}\langle\dbar\ub_\d,\psi\rangle=-\lim_{\d\to0}\langle\ub_\d,\dbar\psi\rangle=\\\nonumber-\lim_{\d\to0}\langle\ub,\dbar\psi_{\d}\rangle=
-\lim_{\d\to0}\langle\ub,(\dbar \psi)_\d\rangle= -\langle \ub,\dbar \psi\rangle= \langle\dbar \ub,\psi\rangle.
\end{gather*}
\hfill $\square$

\section{Proof of the main theorem}

\subsection{Transformations of finite rank forms}\label{6.Sect}
Having now the results on the $\dbar$ equations, in order to establish the finite rank theorem for distributional symbols, we need some elementary facts about the behavior of the finite rank property under certain transformations of the symbol.
\begin{proposition}\label{6.prop.multi}
Let $\Fb\in \L_{q,l}'$ for some $q>1, l\ge0$, $\Fb=\o F$,  and $p(\bar{z})$ be an anti-analytical  polynomial. Suppose that the infinite matrices $\PF(F)$, $\KF(F)$ have  finite rank $\le N$. Then the distribution $\Gb=p(\bar{z})\Fb$ belongs to $\L_{q',l}'$ for any $q'<q$, and the infinite matrices $\PF(G)$, $\KF(G)$, with $G=\o^{-1}\Gb=p(\bar{z})F$, have finite rank, not greater than the $N$.
\end{proposition}
\begin{proof}
 The first statement follows from the fact that if a function $\psi(z)$ belongs to $\L_{q',l}$ then $p(\bar{z})\psi(z)\in \L_{q,l}$ as soon as $q'<q$, and $\|p(\bar{z})\psi(z)\|_{\L_{q,l}}\le C(p)\|\psi\|_{\L_{q',l}}$. Therefore,
\begin{equation*}
    |\langle p(\bar{z}) \Fb, \psi\rangle|=|\langle \Fb,p(\bar{z})\psi\rangle|\le \|\Fb\|_{\L_{q,l}'}\|p(\bar{z})\psi\|_{\L_{q,l}}\le C\|\Fb\|_{\L_{q,l}'}\|\psi\|_{\L_{q',l}}.
\end{equation*}
As for the finite rank property, we consider it on polynomials $z^k,z^{k'}$:
\begin{gather*}
    \tb_{G}(z^k,z^{k'})=\langle \o p(\bar{z})F, z^k\bar{z}^{k'}\rangle=\langle \Fb, p(\bar{z})z^k\bar{z}^{k'}\rangle=\\ \nonumber\tb_F(z^k,\overline{p(\bar{z})}\bar{z}^{k'})=\sum_{j=1}^N(f_j,\overline{p(\bar{z})}v)(u, g_j)=\sum_{j=1}^N(p(\bar{z})f_j,v)(u, g_j).
\end{gather*}
This means that the rank of $\tb_{G}$ is not greater than the rank of $\tb_F$.
\end{proof}
\begin{proposition}\label{6.prop.dbar}
Let $\Fb\in \L_q'$ for some $q>1$, and  $\Fb=\dbar \Gb$, with $\Gb\in \L_{q'}', \ q'>1$. Set $F=\o^{-1} \Fb$, $G=\o^{-1} \Gb$. Suppose that the sesquilinear form $\tb_{F}$ has finite rank on reproducing kernels (or, what is the same, on polynomials). Then the same is correct for the sesquilinear form $\tb_G$, moreover the rank of $\tb_G$ is not greater than the rank of $\tb_F$.
\end{proposition}
\begin{proof} We have,
\begin{equation*}
    \tb_{F}(z^k,z^{k'})=\langle F, z^k \overline{z}^{k'} \rangle=\langle \dbar G, z^k \overline{z}^{k'} \rangle=- \langle G, \dbar(z^k \overline{z}^{k'}) \rangle=-k'\langle  G, z^k \overline{z}^{k'-1}\rangle.
\end{equation*}
So, the infinite matrix $\tb_G(z^k,z^{k'-1})$, $k'\ge1$ is obtained from the matrix $\tb_F(z^k,z^{k'})$,  by means of the multiplication of the columns by the numbers $(-k')^{-1}$. Such operation cannot increase the rank of the matrix. By Proposition \ref{3.converseFRProp}, this implies that the rank of the sesquilinear  form $\tb_G$ is not greater than $\rank(\tb_F)$.
\end{proof}
\subsection{Reduction to functions}\label{Sect.Reduct}
In this subsection we prove that the procedure of solving the $\dbar$ equation, described in Sections \ref{4.dbar functions}, \ref{5.dbar Distr}, leads ultimately to a distribution which actually is a function.

We will use the superscript $\bot$ to denote the subspace consisting of distributions orthogonal to analytical polynomials; say, $\Dc_q'^{\bot}$ consists of distributions $\hb\in \Dc_q'$ for which $\langle\hb,z^k\rangle=0,\ k=0,\dots$.

For a distribution $\hb\in \L_{q,l}'^{\bot}$, $q>1$ , $h=\o^{-1}\hb$, we denote by $\dbar^{-1}\hb$ the solution $\gb$ the $\dbar$ equation $\dbar \gb =\hb$ established by Theorem \ref{4.th.dbar} i.e., belonging to $\L_{q_1,l+1
}'$, $q_1<q/e$ (this solution is, obviously, unique.)
\begin{definition}\label{Int.distr}Let $h$ be a distribution in $\Dc_{q}'$. We say that the distribution $\hb$ is $K-$integrable if there exists a collection of distributions $\hb_k \in \L_{q_k,l}'$, $k=0,1,\dots,K$,  $q_k>0$, $\hb_0=\hb$, $q_0=q$,  and a collection of  \emph{anti-analytical} polynomials $p_k(\bar{z})$,  such that
\begin{equation*}\label{Red}
    \dbar \hb_{k+1}=p_k \hb_k.
\end{equation*}
\end{definition}
It stands to reason that  the condition that a distribution $\hb$ is $K$- integrable implies that each of distributions $p_k\hb_k$ belongs to $\Dc_{q_k'}'^{\bot}$, for any $q_k'<q_k$.
\begin{proposition}\label{prop.integr} Suppose that a distribution $\hb\in \L_{q,l}'^{\bot}$ is $K$-integrable, for certain sufficiently large $K$ (depending on $l,q$). Then the distribution $\hb_K$ is, in fact, a function and the action of this distribution on functions in $\L_{q_N,l}$ is in the natural way,
\begin{equation*}
    \langle \hb_K, u\rangle=\int \hb_K(z) u(z)d\l(z).
\end{equation*}
\end{proposition}

The proof of Proposition \ref{prop.integr} is based upon the analysis of solutions for the $\dbar-$ equation.
\begin{definition}Let $r=l+\t$ be a positive noninteger, $l=[r]$ be the entire part of $r$, $0<\t<1$. We denote by $\L_{q,r}$ the space of functions $\psi\in \L_{q,l}$ such that
\begin{equation*}\label{def.holder}
    e^{-q|z|^2}\sup_{|z_1|,|z_2|<|z|,z_1\ne z_2, |z_1-z_2|<\frac12}\sup_{|\a|\le l}\frac{|\psi^{(\a)}(z_1)- \psi^{(\a)}(z_1)}{|z_1-z_2|^\t}=o(1), |z|\to\infty.
\end{equation*}
\end{definition}
Further on, by \emph{H\"ormander solution} $\overline{\partial}^{-1}f$ of the $\dbar$-equation $\dbar v=f$ will be denoted an arbitrary solution of this equation  provided by Theorem 4.4.2 in \cite{Hor}, reproduced in the Introduction, with $\Wb(z)=q|z|^2$ .

The following lemma deals with local H\"older estimates of H\"ormander solutions.
\begin{lemma}\label{hormLem} Let $f\in\L_{q,0}$ for some $q>0$, $v$ be a H\"ormander  solution, $\dbar v=f$. Then  for any $\g$, $0<\g<1$,

(i): for  any points $z_1,z_2, |z_1-z_2|<\frac12$, one has
\begin{equation*}\label{Horm.Est}
    |v(z_1)-v(z_2)|\le C(\a,f)|z_2-z_1|^\g(\max_{|\z-z_1|\le 1}|f(\z)|+\max_{|\z-z_1|<1}|v(\z)|);
\end{equation*}
(ii) if $f\in \L_{r,q}$, then $v\in \L_{r+\t, q'}$, for any $q'>q$ and $\t\in(0,1)$.
\end{lemma}

\begin{proof}Part (i). We fix $z_1$ and set
\begin{equation}\label{hormLem1}
w(z)=-\frac{1}{\pi}\int_{|\z-z_1|\le2}\frac{f(\z)}{\z-z_1}d\l(\z).
\end{equation}
Then, for $|z-z_1|<2$, we have
\begin{equation*}\label{hormLem2}
    \dbar w(z)=f(z).
\end{equation*}
On the other hand, for all $z$,
\begin{equation*}\label{hormLem3}
    \dbar v(z)=f(z).
\end{equation*}
Thus, the difference $\vf=v-w$ satisfies for $|z-z_1|<2$ the equation $\dbar(v-w)=0$, therefore the function $\varphi=v-w$ is analytical for $|z-z_1|<2$. By the Cauchy theorem,
\begin{equation}\label{hormLem4}
    \varphi(z)=\frac{1}{2\pi i}\int\limits_{|\z-z|=1}\frac{v(\z)-w(\z)}{\z-z}d\z.
\end{equation}
By evaluating  the integrals in \eqref{hormLem1} and in \eqref{hormLem4}, we obtain estimates for the functions  $w$ and $\varphi$:
 \begin{equation*}\label{hormLem5}
    |w(z)|\le 2 \max_{|\z-z_1|\le1}|f(\z)|, \ |z-z_1|<1;
 \end{equation*}
\begin{equation}\label{hormLem6}
    |w(z_2)-w(z_1)|\le C_{\g}|z_2-z_1|^{\g}\max_{|\z-z_1|\le1}|f(\z)|, \ |z-z_1|<1;
\end{equation}
\begin{equation}\label{hormLem7}
    |\varphi(\z)|\le  2\max_{|\z-z_1|\le1}|f(\z)|+\max_{|\z-z_1|\le1}|v(\z)|.
\end{equation}
Now \eqref{hormLem4},\eqref{hormLem7} imply
 \begin{equation}\label{hormLem8}
    |\varphi(z_2)-\varphi(z_1)|\le C|z_1-z_2|(\max_{|\z-z_1|\le1}|f(\z)|+\max_{|\z-z_1|\le1}|v(\z)|), \
 \end{equation}
for $ |z_1-z_2|\le \frac12$. Taken together, \eqref{hormLem6} and \eqref{hormLem8} prove part (i) of the lemma.

 Now we pass to part (ii).
 We write \eqref{hormLem6} as
 \begin{equation}\label{hormLem9}
    v(z)=w(z)+\frac{1}{2\pi i}\int\limits_{|\z-z|=1}\frac{v(\z)-w(\z)}{\z-z}d\z.
 \end{equation}
 It is known that if a given function $f$ belongs to $C^l$ then $w$ also belongs to $C^l$ and, moreover,
 \begin{equation}\label{hormLem10}
    D^\a  w(\z)=-\frac{1}{\pi}\int\limits_{|\z-z|\le1} \frac{D^\a f(\z)}{\z-z}d\l(\z), \ |\a|\le l.\end{equation}
 Therefore,
 \begin{equation}\label{hormLem11}
    |D^{\a}w(z_1)|\le 2 \max_{|\z-z_1|\le1}|D^{\a}f(\z)|, \ |\a|\le l.
 \end{equation}
 By estimating the integral for $\varphi(z)$, we obtain
 \begin{equation}\label{hormLem12}
    |D^{\a}\varphi(z)|\le C_\a
(\max_{|\z-z_1|\le1}|v(\z)|+\max_{|\z-z_1|\le1}|f(\z)|) \end{equation}
 From \eqref{hormLem9}-\eqref{hormLem12} we obtain the statement of the second part of Lemma
\end{proof}
We can give now the proof of Proposition \ref{prop.integr}.
\begin{proof}
For $\hb\in \L_{q,l}'^{\bot}$ we established the existence of $\hb_1\in \L_{q_1, l+1}'$ such that $\dbar \hb_1=\hb$ for any $q_1<q/e$. After the multiplication by the polynomial $p_1(\bar{z})$, we obtain the distribution $p_1\hb_1\in \L_{q_1,l+1}^{'\bot}$, solve the equation $\dbar\hb_2=p_1\hb_1$, $\hb_2\in \L_{q_2,l+2}$ and repeat this construction sufficiently many ($M=2l+2$) times.
 We are going to show that, in fact, the distribution $\hb_M$ is a function in $\L_{q',l}$ for some $q'>2,$ provided $\hb\in \L_{q,l}'$ for sufficiently large $q$.

So we suppose that $q>2e^M$ and $\hb\in \L_{l,q}^{'\bot}$. Then $\hb_M,$ by Theorem \ref{4.th.dbar}, belongs to $ \L_{q',l+M'}$ for any $q'<e^{-M}q$, so, by assumptions on $q$, $q'$ can be chosen greater than 2.

To prove that $\hb_M$ is a function, we fix an arbitrary $R>2$  and choose  a sequence of radii, $R_M>R_{M-1}>\dots>R_1>R_0=R.$ With these radii we associate cut-off functions $\theta_k\in C_0^\infty(B_{R_{k+1}})$, $k=0,1,\dots,M$, so that $\theta_k=1$ on $B_{R_k}$, where $B_R$ stands for the disk, centered at zero, with radius $R$.

We define the following distributions:
\begin{gather}\label{Propintegr1}
    \vb_1=(\theta_1 \hb)*G;\dots\\ \nonumber
    \vb_k=(\theta_k p_k \vb_{k-1})*G;\dots\\ \nonumber
    \vb_M=\theta_M ((\theta_M p_M \vb_{M-1})*G),
\end{gather}
where $G(z)=-\frac1\pi z^{-1}$, i.e., the fundamental solution for $\dbar$.
The convolutions in \eqref{Propintegr1} are understood in the sense of convolution of distributions and are well defined since one of the terms is compactly supported, thanks to the cut-off factors. The distribution $\vb_M$ is, in fact, a continuous function. We give a standard explanation of this fact.
The starting term in the sequence of distributions, $\theta_1 \hb$, being a distribution of finite order  with compact support, belongs to a Sobolev space $H^s$ with some (negative) $s$, $s>-l$. The convolution with $G$ is a solution of the $\dbar$ equation, and by ellipticity of the first order operator $\dbar$, we obtain that $\theta_2 \vb_1$ belongs to the Sobolev space $H^{s+1}$, together with $p_2\theta_2 \vb_1$. On the next step we, again using the ellipticity of $\dbar$, obtain that $p_3\theta_3 \vb_2$ belongs to $H^{s+2}$, and so on, until we arrive to a distribution belonging to the Sobolev space of order larger than 1, where elements are continuous functions.

Now we compare the distributions $\vb_k$ and $\hb_k$. We have equations
\begin{gather*}\label{Propintegr2}
    \dbar \hb_1=\hb_0,\\ \nonumber
    \dbar \vb_1=\theta_1\hb_0.
    \end{gather*}
    Subtracting, we obtain
    \begin{equation*}\label{Propintegr3}
        \dbar(\hb_1-\vb_1)=\hb_0(1-\theta_1).
    \end{equation*}
    The factor $1-\theta_1$ vanishes in the disk $B_0=\{|z|< R\}$ so the support of the distribution $\hb_0(1-\theta_1)$ is disjoint with this disk. This means that for any function $u\in C_0^\infty(B_0)$,
    $\langle\dbar(\hb_1-\vb_1),u\rangle=0$, therefore $\hb_1-\vb_1=\phi_1$ must be an analytical function in $B_0$ (more exactly, in the sense of distributions in $\Dc'(B_0)$, $\chi(z)(\hb_1-\vb_1)=\chi(z)\phi_1$, for any $\chi\in C_0^\infty(B_0)$).
    On the next step, after the multiplication by the polynomial $p_1(\bar{z})$, we have
    \begin{gather*}\label{Propintegr4}
    \dbar \hb_2=p_1(\bar{z})\hb_1,\\ \nonumber
    \dbar \vb_2=p_1(\bar{z})\theta_2 \vb_1.
    \end{gather*}
    We subtract to obtain
    \begin{equation}\label{Propintegr5}
        \dbar(\hb_2-\vb_2)=p_1(\bar{z})(\hb_1-\theta_2 \vb_2).
         \end{equation}
    The right-hand side of \eqref{Propintegr5} coincides in $B_0$ with $p_1(\bar{z})\phi_1(z)$. Therefore the solution $h_2-v_2$ of  the equation \eqref{Propintegr5} must coincide in $B_0$ with a function of the form
    \begin{equation*}\label{Propintegr6}
      \hb_2-\vb_2=p_1^1(\bar{z})\phi_1(z)+\phi_2(z),
    \end{equation*}
    where $p_1^1(\bar{z})$ is a polynomial, a primitive function for $p_1$, and $\phi_2(z) $ is a function analytical in $B_0$.
    We repeat this reasoning $M$ times, obtaining a representation of $\hb_k-\vb_k$ in $B_0$ on each step. Finally, we arrive at the representation
    \begin{equation*}\label{Propintegr7}
        h_M(z)-v_M(z)=\sum_{k=1}^M q_k(\bar{z})\phi_k(z),
    \end{equation*}
    in $B_0$ with some  anti-analytical  polynomials $q_k(\bar{z})$ and analytical functions $\phi_k(z)$.
    It follows from \eqref{Propintegr7} that
\begin{equation*}
    h_M(z)=v_M(z)+\sum_{k=1}^M q_k(\bar{z})\phi_k(z),
\end{equation*}
in $B_0$, and since both terms on the right belong to $C_1(B_0)$ the right-hand side, $\hb_M(z)$ also belongs to $C_1(B_0)$. By arbitrariness of $R$, it follows that $\hb_M(z)$ coincides with the a $C^1$- function everywhere.
\end{proof}
To prove the required estimate for the, now, function $\hb_M(z)$, we the following lemma.

\begin{lemma}\label{perenos}Let $\hb\in \L_{q,l}'^{\bot}$, $f\in\L_{q',l}$, $q'<\frac{q}{e}$. Let $\gb\in \L'_{q',l}$ be the solution of the equation $\dbar \gb=\hb$, obtained in Theorem \ref{4.th.dbar}  and $f_1=\dbar^{-1}f$ be the H\"ormander solution of the equation $\dbar f_1=f$. Then
\begin{equation}\label{perenos1}
    \langle \gb,f\rangle=\langle \hb,f_1\rangle.
\end{equation}
\end{lemma}
\begin{proof} Both parts of \eqref{perenos1} are correctly defined. For the right-hand side, it follows from the inclusions $\gb\in \L'_{q',l}$ and $f\in \L_{q',l}$. For the right-hand side, by Proposition \ref{hormLem}, $f_1\in \L_{q'', l+\a}$ for $q''<q'<q/e$.  Now we have
\begin{equation*}\label{perenos2}
    \langle \hb, f_1\rangle=\langle \dbar \gb, f_1\rangle=\langle \gb,\dbar f_1\rangle=\langle \gb, f\rangle.
\end{equation*}
\end{proof}
We continue the \emph{proof} of Proposition \ref{prop.integr}. Now for a  given $\hb=\hb_0\in  \L'_{q,l} $, $f\in \L_{q',l}$, $q'< qe^{-M-2}$ we construct two sequences. Further on, $\hb_k$ is the solution of the equation $\dbar \hb_k=p_{k-1}(\bar{z})\hb_{k-1}$, $k=1,\dots,M+2$ belonging to $\L_{q_k,l}'$, $q_k<qe^{-k}$; $f_k$ is the H\"ormander solution of the equation $\dbar f_k=p_k(\bar z)f_{k-1}$. Then, applying Lemma \ref{perenos} $M+1$ times, we obtain
\begin{equation*}\label{Perenos3}
    \langle \hb_{M+2}, f\rangle=\langle \hb, f_{M+2}\rangle.
\end{equation*}
We fix $\e>0$ and use Lemma \ref{hormLem} with $\a=1-\frac{1}{2N+3}, \ M=2N+2$, to obtain
\begin{equation*}\label{Perenos4}
    \|f_M\|_{q_M+\e, N}\le \|f\|_{q_{M},0}.
\end{equation*}
Therefore, for $q'<qe^{-M}, $ we have
\begin{equation}\label{Perenos5}
|\langle \hb_{M+2}, f\rangle|=|\langle \hb,f_{M+2}\rangle|\le \|f_{M+2}\|_{q,N}\le C \|f\|_{q'-\e,0}.
\end{equation}
Since $\hb_{M+2}$ is a continuously differentiable function, the inequality \eqref{Perenos5} means that
\begin{equation}\label{Perenos6}
    \int\limits_{\C}\hb_{M+2}(z)f(z)d\l(z)\le C\|f\|_{q'-\e,0}.
\end{equation}
for any $f\in\L_{q'-\e,0}$, this means, for any  function $f$ with the prescribed  growth rate.
Let $\O=\{z: |\hb_{M+2}(z)|>0\}$. By the smoothness of $\hb_{M+2}$, $\l(\partial \O)=0$. We choose $f(z)$ in the following way. In  $\O$ we take $|f(z)|=e^{(q'-2\e)|z|^2}$ and set $f(z)=0$ outside $\O$. The argument of $f(z)$ in $\O$ is chosen so that $|\hb_{M+2}(z)||f(z)|=\hb_{M+2}(z)f(z)$. With such choice of $f$, \eqref{Perenos6} gives
\begin{equation*}\label{Perenos7}
    \int\limits |\hb_{M+2}(z)|e^{(q'-2\e')|z|^2}d\l(z)\le C_\e.
\end{equation*}
This inequality implies the statement of the proposition.

\subsection{The main theorem}\label{last}
Finally, we establish our main result about finite rank operators.
\begin{theorem}\label{final theorem}
Let $\Fb$ be a distribution in $\L_{q}'$, $q>0$; $F=e^{|z|^2}\Fb$. Suppose that the matrices $\KF(F), \PF(F)$ have finite rank.  Then $F$ is a finite linear combination of $\delta$-distributions at no more than $N$ points and their derivatives:
\begin{equation*}\label{fin.1}
     \Fb=\sum_{j=1}^{N_0} \Lc_j\d_{z_j},
\end{equation*}
where $N_0\le N$, $\Lc_j$ are differential operators.
\end{theorem}
\begin{proof} Let $l$ be the order of the distribution $\Fb$: the distribution can be extended to the one in $\L'_{q,l}$ for some $l$. We start by applying the scaling transformation $S_t$ with sufficiently large $t$: by Proposition \ref{PropBdd}, this large $t$ can be chosen in such way that the corresponding Toeplitz form is bounded in the Fock space. Even more, $t$ should be taken so large that the transformed distribution $F_t$ belongs to $\L_{Q,l}'$ with $Q> 2e^{2l}$.

The property of having finite rank for the matrix, and the value of such rank, are invariant under scaling, by Proposition \ref{3.prop.invar}. So we can, from the very beginning, suppose that the initial symbol $F$ possesses the above properties, and we will write  $q$ instead of $Q$ further on.

Now, since the sesquilinear form $\tb_{F}$ is bounded, the corresponding Toeplitz operator $\Tb(F)$  is bounded, and, by Proposition \ref{3.converseFRProp}, has finite rank $N$ in the sense of  \eqref{2.FRforms}.

Further on, the reasoning  follows  the proof of Theorem 3.1 in \cite{AlexRoz}.

We consider the infinite matrix $\PF(F)$. By our assumptions, this matrix has rank  $N$. This means that the columns of this matrix are linearly dependent, so one can find coefficients $\g_j, j=0,j=0,\dots,N-1$, such that  for the polynomial $p_1(z)=\sum_{j=0}^{N-1}\g_jz^j$, the equality
\begin{equation*}\label{fin.2}
    \tb_F(z^k,p_1(z))\equiv \langle \Fb,z^k \overline{p_1(z)}\rangle=\langle  \overline{p_1(z)} \Fb,z^k\rangle=0
\end{equation*}
holds for all $k\in \Z_+$. In our notations, this means that the distribution $\overline{p_1(z)} \Fb$ belongs to the space $\L_{q-\e,l}^\bot$, with any $\e>0$. Therefore, we can apply to this distribution Theorem \ref{4.dbar functions}, which establishes the existence of a distribution $\Fb_1\in \L'_{q/e-\e,l}$ solving the equation $\dbar \Fb_1=\overline{p_1(z)} \Fb.$ By Proposition \ref{6.prop.dbar}, the  sesquilinear form $\tb_{F_1}$, $F_1=\o(z)^{-1}\Fb_1$ has finite rank, not larger than $N$,
and   it is still bounded in the Fock space $\Bc$.

We can repeat the reasoning above with the distribution $\Fb_1$. The infinite matrix $\langle \Fb_1, z^k\bar{z}^{k'}\rangle$ has rank not larger than $N$, therefore there exists a polynomial $p_2(\bar{z})$ such that  $\langle  p_2(\bar{z})\Fb_1,z^k\rangle=0$, $k\in \Z_+$. The distribution $  p_2(\bar{z})\Fb_1$ thus satisfies the conditions of Theorem \ref{4.th.dbar}, and therefore  there exists $\Fb_2\in \L'_{qe^{-2}-\e, l}$ such that $\dbar \Fb_2=  p_2(\bar{z})\Fb_1.$

After some finite number of such steps, we arrive at a distribution $\Fb_M=\o F_M$, which, in fact, is  a function, for which the Toeplitz operator has finite rank. By Theorem 3.1 in \cite{BorRoz}, this function must be zero.
Therefore, the distribution $\dbar{\Fb_M}=p_{M-1}(\bar{z})\Fb_{M-1}$
equals zero. This can be only if the support of $\Fb_{M-1}$ is contained in the set of zeros of the polynomial $p_{M-1}(\bar{z})$, therefore, $\Fb_{M-1}$ is a finite sum of $\d$-distributions and their derivatives. The same holds for the distribution $\dbar{\Fb_M}=p_{M-1}(\bar{z})\Fb_{M-1}$, and thus the support of $\Fb_{M-2}$ consists of a finite set of points, and, as before, $\Fb_{M-2}$ is a finite sum of $\d$-distributions and their derivatives. After $M$ such steps, we return to the initial distribution $\Fb$, which, again is a finite combination of $\d$-distributions and their derivatives.
Finally, to show that there are no more than $N$ points one can repeat the interpolation argument from \cite{AlexRoz}.
\end{proof}

\end{document}